\documentclass[12pt]{article}
\usepackage{latexsym}
\usepackage{times}
\usepackage{graphicx,color}
\usepackage{amsmath,amssymb, amsfonts,amsthm}
\catcode`@=11
\def\@sect#1#2#3#4#5#6[#7]#8{\ifnum #2>\c@secnumdepth
     \def\@svsec{}\else 
     \refstepcounter{#1}\edef\@svsec{\csname the#1\endcsname{.}\hskip 1em }\fi
     \@tempskipa #5\relax
      \ifdim \@tempskipa>\z@ 
        \begingroup #6\relax
          \@hangfrom{\hskip #3\relax\@svsec}{\interlinepenalty \@M #8\par}
        \endgroup
       \csname #1mark\endcsname{#7}\addcontentsline
         {toc}{#1}{\ifnum #2>\c@secnumdepth \else
                      \protect\numberline{\csname the#1\endcsname}\fi
                    #7}\else
        \def\@svsechd{#6\hskip #3\@svsec #8\csname #1mark\endcsname
                      {#7}\addcontentsline
                           {toc}{#1}{\ifnum #2>\c@secnumdepth \else
                             \protect\numberline{\csname the#1\endcsname}\fi
                       #7}}\fi
     \@xsect{#5}}

\@addtoreset{equation}{section}
\newcommand{\pend}{\hspace*{\fill} $\square$}

\newtheorem{lemma}{Lemma}[section]
\newtheorem{theorem}[lemma]{Theorem}
\newtheorem{corollary}[lemma]{Corollary}

\theoremstyle{definition}

\newtheorem{remark}{Remark}

\begin{document}

\title{A Note on a Nearly Uniform Partition into 
Common Independent Sets of Two Matroids}

\author{{ Satoru Fujishige}\footnote{
RIMS, 
Kyoto University, Kyoto 606-8502, Japan. 
\ \ E-mail: {fujishig@kurims.kyoto-u.ac.jp}
},\ \ 
{ Kenjiro Takazawa}\footnote{
Hosei University, Tokyo 184-8584, Japan.
\ \ E-mail: {takazawa@hosei.ac.jp}},\ \ 
and Yu Yokoi\footnote{
National Institute of Informatics, Tokyo 101-8430, Japan. 
\ \ E-mail: {yokoi@nii.ac.jp}
}
}

\date{September, 2019}

\maketitle

\begin{abstract}

The present note is a strengthening of a recent paper by K. Takazawa 
and Y. Yokoi 
(A generalized-polymatroid approach to disjoint common independent sets 
in two matroids, {\it Discrete Mathematics} (2019)). 
For given two matroids on $E$, under the same assumption in their paper 
to guarantee the existence of a partition of $E$ into $k$
common independent sets of the two matroids, we show that 
there exists a nearly 
uniform partition $\mathcal{P}$ of $E$ into $k$ common independent sets, 
where the difference of the cardinalities of any two sets 
in $\mathcal{P}$ is at most one.

\end{abstract}


\section{Introduction}
\label{sec:1}

K. Takazawa and Y. Yokoi \cite{TakazawaYokoi2018} have very recently 
showed a new approach to the problem of partitioning 
the common ground set of two matroids 
into common independent sets by means of generalized polymatroids. 
They successfully give a unifying view on some results of J. Davies 
and C. McDiarmid \cite{DaviesMcDiarmid1976} and D. Kotlar and R. Ziv 
\cite{KotlarZiv2005} and extend them by the generalized-polymatroid 
approach.

A partition $\mathcal{P}$ of a finite nonempty set $E$ is called 
{\it nearly uniform} if 
 the cardinality difference of every pair of sets 
in $\mathcal{P}$ is at most one.
Researchers' attention has been drawn to the existence of a nearly 
uniform partition of the ground set of a combinatorial system 
into disjoint objects of the system 
such as branchings (\cite[Sec.~53.6]{Schrijver2003}) and matchings 
(\cite{DaviesMcDiarmid1976,KiralyYokoi2018}).
In the present note we show that the generalized-polymatroid 
approach in \cite{TakazawaYokoi2018} reveals the existence of a nearly 
uniform partition $\mathcal{P}$ of $E$ into common independent sets 
of two matroids under the same assumption in \cite{TakazawaYokoi2018}.

In Section~\ref{sec:2} we describe the result of Takazawa and 
Yokoi \cite{TakazawaYokoi2018} in a general form, which is basically
a dynamic programming formulation. Then, in Section~\ref{sec:nearly},
under the same assumption 
in the paper \cite{TakazawaYokoi2018}
to guarantee the existence of a partition of $E$ into $k$
common independent sets of the two matroids, 
we show that there exists a nearly 
uniform partition $\mathcal{P}$ of $E$ into $k$ common independent sets, 
where the difference of the cardinalities of any two sets 
in $\mathcal{P}$ is at most one.
Section~\ref{sec:4} gives some concluding remarks.

\section{The Generalized-Polymatroid Approach of Takazawa and Yokoi}
\label{sec:2}

We follow the definitions and notation given in \cite{TakazawaYokoi2018}
(and in our Appendix).
A brief survey about fundamental facts about matroids,  
polymatroids, generalized polymatroids, and submodular/supermodular functions 
is given in the appendix for readers' convenience. 
Also see \cite{Frank2011,Fuji2005,Oxley2011,Schrijver2003,Welsh1976}. 

Let $E$ be a nonempty finite set. For each $i=1,2$ 
let ${\bf M}_i = (E,\mathcal{I}_i)$ be a matroid on $E$ 
with $\mathcal{I}_i\subseteq 2^E$ being a family of independent sets.
For a given positive integer $k\ge 2$  
let ${\bf M}_i^k=(E,\mathcal{I}_i^k)$ be the union matroid of $k$ copies 
of ${\bf M}_i$ for each $i=1,2$, and we assume 
that $E\in\mathcal{I}_1^k\cap\mathcal{I}_2^k$. 

Now, consider the problem of partitioning the ground set $E$ of 
the two matroids 
${\bf M}_i$ $(i=1,2)$ into $k$ common independent sets as follows:
\begin{description}
\item[(P):] Find a partition $\mathcal{P}=\{X_1,\cdots,X_k\}$ 
of $E$ into $k$ disjoint subsets $X_j\subseteq E$ $(j=1,\cdots, k)$
such that $X_j\in\mathcal{I}_1\cap\mathcal{I}_2$ for all $j=1,\cdots,k$.
\end{description}
Here we allow empty component $X_j=\emptyset\in\mathcal{I}_1\cap\mathcal{I}_2$,
just by a technical reason for the arguments in the sequel.
(It should be noted that if we can partition $E$ into $k$ 
possibly empty common independent
sets, then we can partition $E$ into $k$ nonempty common independent
sets when $k\le|E|$.) 

Let $\rho$ be the rank function of ${\bf M}=(E,\mathcal{I})$, 
$\mathcal{I}^k$ the union matroid ${\bf M}^k$ of 
$k$ copies of ${\bf M}=(E,\mathcal{I})$,
$\rho^{k}$ the rank function
of the union matroid ${\bf M}^k=(E,\mathcal{I}^k)$,
 $\rho^\#$ the dual supermodular function of $\rho$, 
${\rm P}(\rho)$ the submodular polyhedron associated with submodular $\rho$, 
and ${\rm P}(\rho^\#)$ the supermodular polyhedron 
associated with supermodular $\rho^\#$ (see Appendix).
Also for any family $\mathcal{F}$ of subsets of $E$ denote 
by ${\rm Conv}(\mathcal{F})$  
the convex hull of characteristic vectors $\chi_X\in\mathbb{R}^E$ for
all $X\in\mathcal{F}$.

\begin{theorem}[\cite{TakazawaYokoi2018}]\label{th:TY1}
Let ${\bf M}=(E,\mathcal{I})$ be a matroid with $E\in\mathcal{I}^k$. 
Define
\begin{equation}\label{eq:fa}
 \mathcal{F}=\{X\mid \mbox{$X\in\mathcal{I}$, $E\setminus X$ can be partitioned into $k-1$ sets in $\mathcal{I}$}\}.
\end{equation}
Then we have
\begin{equation}\label{eq:fc}
  {\rm Conv}(\mathcal{F})
  ={\rm P}(\rho)\cap{\rm P}((\rho^{k-1})^\#)\subseteq [0,1]^E.
\end{equation}
\end{theorem}

\begin{remark}\label{re:1} 
Note that $E\setminus X$ can be partitioned into $k-1$ sets in $\mathcal{I}$
if and only if $X$ is a co-spanning set of the union matroid ${\bf M}^{k-1}
=(E,\mathcal{I}^{k-1})$ (see Appendix). 
In Theorem \ref{th:TY1} the right-hand side of 
(\ref{eq:fc}) is the intersection of the submodular polyhedron 
${\rm P}(\rho)$ and the supermodular polyhedron 
${\rm P}((\rho^{k-1})^\#)$, which is nonempty 
by the assumption that $E\in\mathcal{I}^k$ (implying 
$(\tfrac{1}{k},\cdots,\tfrac{1}{k})\in 
{\rm P}(\rho)\cap{\rm P}((\rho^{k-1})^\#)$) and is integral.

 Hence a set $X\in\mathcal{F}$ can be 
found efficiently and we can further apply this process 
for $k\gets k-1$,  $E\gets E\setminus X$
and ${\bf M}\gets {\bf M}^E$ (the restriction of ${\bf M}$ on the updated $E$).
We can repeat this process to obtain a partition $\{X_1,\cdots,X_k\}$ 
of $E$ into $k$ independent sets $X_j\in\mathcal{I}$ $(j=1,\cdots,k)$.
Though we have more direct, efficient algorithms to find a partition 
$\{X_1,\cdots,X_k\}$ 
of $E$ into independent sets $X_j\in\mathcal{I}$ $(j=1,\cdots,k)$, 
Theorem \ref{th:TY1} gives a basis for the generalized-polymatroid 
approach to Problem ${\bf (P)}$ of Takazawa and Yokoi \cite{TakazawaYokoi2018}.
\pend
\end{remark}

Now we have the following theorem, based on Theorem~\ref{th:TY1}.

\begin{theorem}[\cite{TakazawaYokoi2018}]\label{th:TY2}
Consider two matroids ${\bf M}_i$ $(i=1,2)$ such  
that $E\in\mathcal{I}_1^k\cap\mathcal{I}_2^k$.
Let $\ell\in\{0,1,\cdots,k-1\}$ and let $\{X_1,\cdots,X_{\ell}\}$ be a set of 
disjoint $\ell$ common independent sets of ${\bf M}_i$ $(i=1,2)$.\footnote{
When $\ell=0$, regard $\{X_1,\cdots,X_{\ell}\}$ as an empty family and
$\bigcup_{j=1}^{\ell} X_j=\emptyset$.}
Putting $F=E\setminus \bigcup_{j=1}^{\ell} X_j$, define
for each $i=1,2$ 
\begin{equation}\label{eq:faa}
 \mathcal{F}^{\ell}_i(F)=\{X\subseteq F\mid 
        \mbox{$X\in\mathcal{I}_i$, $F\setminus X$ can be partitioned into $k-\ell-1$ sets in $\mathcal{I}_i$}\}.
\end{equation}
Then we have
\begin{equation}\label{eq:fd}
{\rm Conv}(\mathcal{F}_1^{\ell}(F)\cap\mathcal{F}_2^{\ell}(F))\subseteq
{\rm P}(\rho_1^F)\cap{\rm P}(((\rho_1^F)^{k-\ell-1})^\#)
\cap{\rm P}(\rho_2^F)\cap{\rm P}(((\rho_2^F)^{k-\ell-1})^\#),
\end{equation}
where $\rho_i^F$ is the rank function of the restriction of ${\bf M}_i$
on $F$.

If the intersection of the four polyhedra on the right-hand side of 
{\rm (\ref{eq:fd})} contains an integral point, i.e., a characteristic
vector $\chi_{X^*}$ of some $X^*\subseteq F$, then we have  
$X^*\in\mathcal{F}_1^{\ell}(F)\cap\mathcal{F}_2^{\ell}(F)$.
In particular, if the intersection of the four polyhedra 
on the right-hand side of {\rm (\ref{eq:fd})} is integral, then 
the inclusion relation {\rm (\ref{eq:fd})}
 holds with equality and 
there exists a set $X\in\mathcal{F}_1^{\ell}(F)\cap\mathcal{F}_2^{\ell}(F)$.
\end{theorem}

\begin{remark}\label{re:2}
Note that 
${\rm P}(\rho_1^F)\cap{\rm P}(((\rho_1^F)^{k-\ell-1})^\#)$ and 
${\rm P}(\rho_2^F)\cap{\rm P}(((\rho_2^F)^{k-\ell-1})^\#)$ are integral 
for any matroids ${\bf M}_i$ $(i=1,2)$,  
due to Theorem \ref{th:TY1}, 
but their intersection does not necessarily contains an integral point. 
Since by the assumption that $E\in\mathcal{I}_1^k\cap\mathcal{I}_2^k$
the vector $(\tfrac{1}{k},\cdots,\tfrac{1}{k})$ belongs to the 
intersection of the four polyhedra in (\ref{eq:fd}) for $\ell=0$, 
if the intersection of the four polyhedra is integral (or more generally 
contains an integral point), 
there exists a set 
$X_1\in\mathcal{F}_1^{0}(F)\cap\mathcal{F}_2^{0}(F)$. 
Then for $F=E\setminus X_1$ we can apply the same arguments to find 
$X_2\in\mathcal{F}_1^{1}(F)\cap\mathcal{F}_2^{1}(F)$, and repeatedly
carry out this process to find a desired partition $\{X_1,\cdots,X_k\}$
into common independent sets. Also see the proof of Theorem~\ref{th:e1}.
\pend
\end{remark}

\begin{remark}\label{re:3}
Takazawa and Yokoi \cite{TakazawaYokoi2018} considered the case when 
the intersection of the first two polyhedra in (\ref{eq:fd}) is 
a generalized polymatroid and so is that of the last two. 
They showed that such matroids are given by  
laminar matroids (special gammoids) and the matroids 
without $(k+1)$-spanned elements 
considered by Kotlar and Ziv~\cite{KotlarZiv2005}. 
Since the nonempty intersection of two integral generalized polymatroids
is integral, 
they thus showed that every pair of matroids from among 
laminar matroids and the matroids without $(k+1)$-spanned elements 
considered by Kotlar and Ziv~\cite{KotlarZiv2005}
satisfies the assumptions of Theorem~\ref{th:TY2}, 
for which our problem ${\bf (P)}$ is 
efficiently solvable.
\pend
\end{remark}

\section{Nearly Uniform Partitions}\label{sec:nearly}

Let us further examine the generalized-polymatroid approach of Takazawa 
and Yokoi given by Theorems~\ref{th:TY1} and \ref{th:TY2}
for the problem of partitioning two matroids into common independent
sets.

\begin{theorem}\label{th:e1}
Let  ${\bf M}_i = (E,\mathcal{I}_i)$ $(i=1,2)$ be matroids and 
$k$ be a positive integer such that $E\in\mathcal{I}_1^k\cap\mathcal{I}_2^k$.
If every ${\rm P}(\rho_i^F)\cap{\rm P}(((\rho_i^F)^{k-\ell-1})^\#)$ 
for $i=1,2$ 
appearing in {\rm (\ref{eq:fd})} in Theorem {\rm \ref{th:TY2}} is 
an integral generalized polymatroid,
then there exists a nearly uniform partition 
of $E$ into $k$ common independent sets of 
${\bf M}_i = (E,\mathcal{I}_i)$ $(i=1,2)$.
\end{theorem}
\begin{proof}
Put $\lambda=|E|/k$ and define $\lambda^+=\lceil\lambda\rceil$ 
and $\lambda^-=\lfloor\lambda\rfloor$. 
It follows from Theorem~\ref{th:TY2} and  the assumptions of the present 
theorem that if we find $X_j$ for $j=1,\cdots,\ell$ by the procedure 
described in Remark~\ref{re:1}, then 
for each $i=1,2$ the polyhedron given by 
\begin{equation}
 {\rm P}(\rho_i^F)\cap{\rm P}(((\rho_i^F)^{k-\ell-1})^\#)
 \cap \{x\in\mathbb{R}^F\mid \lambda^-\le x(F)\le\lambda^+\} 
\end{equation}
is an integral generalized polymatroid (due to Fact 3 in Appendix) 
and contains the uniform vector 
$(\tfrac{1}{k-\ell},\cdots,\tfrac{1}{k-\ell})$ in $\mathbb{R}^F$
and hence there exists  
a set $X\in \mathcal{F}^\ell_1(F)\cap\mathcal{F}^\ell_2(F)$ with 
$\lambda^-\le |X|\le \lambda^+$.\footnote{Note that
we have initially $\lambda^-\le |E|/k\le\lambda^+$ and hence 
$\lambda^-\le|X_1|\le\lambda^+$, and then we have 
$\lambda^-\le(|E|-|X_1|)/(k-1)\le\lambda^+$. So 
we can show by induction that  for $F$ in (\ref{eq:fd})
we have $\lambda^-\le |F|/(k-\ell)\le\lambda^+$ for $\ell=0,1,\cdots,k-1$.} 
Hence there exists a partition $\{X_1,\cdots,X_k\}$ of $E$ 
into common independent sets of ${\bf M}_i$ $(i=1,2)$ such that 
$\lambda^-\le |X_j|\le\lambda^+$ for all $j=1,\cdots,k$.
\end{proof}

\begin{remark}\label{re:4}
Since laminar matroids and the matroids considered in \cite{KotlarZiv2005}
satisfy the assumptions required in Theorem \ref{th:e1} as shown by
Takazawa and Yokoi \cite{TakazawaYokoi2018}, for every pair of 
such matroids ${\bf M}_i=(E,\mathcal{I}_i)$ $(i=1,2)$ 
with $E\in\mathcal{I}_1^k\cap\mathcal{I}_2^k$ 
there exists a nearly uniform partition of $E$ into $k$ 
common independent sets.
\pend
\end{remark}

\begin{remark}\label{re:5}
Theorem~\ref{th:e1} can be given in a more general form as described in 
Theorem~\ref{th:TY2}. 
That is, it suffices to impose that the intersection of the four 
polyhedra in (\ref{eq:fd}) and 
$\{x\in\mathbb{R}^F\mid \lambda^-\le x(F)\le\lambda^+\}$ 
with $\lambda^-=\lfloor|E|/k\rfloor$ and $\lambda^+=\lceil|E|/k\rceil$ 
contains an integral point.
\pend
\end{remark}

For {\it general} matroids  ${\bf M}_i=(E,\mathcal{I}_i)$ $(i=1,2)$ 
we also have the following.
Define for each $i=1,2$
\begin{equation}\label{eq:f1}
 \mu^*_i=\min\{\mu\in\mathbb{Z}_{>0}\mid E\in\mathcal{I}_i^\mu\},
\end{equation}
which is the covering index for matroid ${\bf M}_i=(E,\mathcal{I}_i)$
$(i=1,2)$. 
A {\it subpartition} of $E$ is a set of disjoint subsets of $E$.

\begin{theorem}\label{th:e3}
Let  ${\bf M}_i = (E,\mathcal{I}_i)$ $(i=1,2)$ be arbitrary matroids and 
$k$ be a positive integer such that $E\in\mathcal{I}_1^k\cap\mathcal{I}_2^k$.
Suppose that $\mu^*_1\le\mu^*_2< k$.
Then there exists a nearly uniform subpartition 
$\{X_1,\cdots,X_{k-\mu^*_2-1}\}$ of $E$ such that 
\begin{itemize}
\item $X_\ell\in\mathcal{I}_1\cap\mathcal{I}_2$\ \ for\ \ 
$\ell=1,\cdots, k-\mu^*_2-1$,
\item $E\setminus(X_1\cup\cdots\cup X_{k-\mu_2^*-1})
\in \mathcal{I}_1^{\mu^*_2+1}\cap \mathcal{I}_2^{\mu^*_2+1}$.
\end{itemize}
\end{theorem}
\begin{proof}
For $\ell=1,\cdots,k-\mu_2^*-1$, 
under the assumption of the present theorem, for each $i=1,2$ we have 
$\emptyset \in \mathcal{F}_i^\ell(F)$ in (\ref{eq:faa}), so that  
 $\mathcal{F}_i^\ell(F)$ is actually $\mathcal{I}_i$.
Hence the argument in the proof of Theorem \ref{th:TY2} can be 
adapted for obtaining a nearly uniform subpartition 
$\{X_1,\cdots,X_{k-\mu_2^*-1}\}$ of $E$ satisfying the conditions of 
the present theorem.
\end{proof}

Similarly we can show the following, a corollary of Theorem~\ref{th:TY1},
which may be folklore.

\begin{corollary}\label{cor:2}
For an arbitrary matroid ${\bf M}=(E,\mathcal{I})$ with $E\in\mathcal{I}^k$ 
there exists a nearly uniform partition of $E$ into $k$ independent sets
of ${\bf M}$.
\end{corollary} 

It should be noted that Corollary \ref{cor:2} holds for any general matroid
${\bf M}=(E,\mathcal{I})$ with $E\in\mathcal{I}^k$, but for two matroids 
${\bf M}_i=(E,\mathcal{I}_i)$ $(i=1,2)$ 
with $E\in\mathcal{I}_1^k\cap\mathcal{I}_2^k$ we need additional conditions
to guarantee the existence of a nearly uniform partition of $E$ into
common independent sets, in general, such as those 
given in Theorem~\ref{th:e1}.

\section{Concluding Remarks}\label{sec:4}

We have shown that under the same assumption in \cite{TakazawaYokoi2018}
that makes the generalized-polymatroid approach of Takazawa and Yokoi 
work, there also exists a nearly uniform partition into common 
independent sets.

It is interesting to identify the class of pairs of matroids 
for which every intersection of the four polyhedra in (\ref{eq:fd}) 
is integral and computationally tractable, 
 which is left open.
Besides the way of using generalized polymatroids in \cite{TakazawaYokoi2018}
there may be the case when the intersection of the first and the fourth 
polyhedra in (\ref{eq:fd}) is a generalized polymatroid and so is 
the intersection of the second and the third.

\section*{Acknowledgements}

S. Fujishige is supported 
by JSPS KAKENHI Grant Numbers 19K11839, Japan.
K. Takazawa is partially supported by
JST CREST Grant Number JPMJCR1402 and 
 JSPS KAKENHI Grant Numbers JP16K16012 and JP26280004, Japan.
Y. Yokoi is supported by JST CREST Grant Number JPMJCR14D2 and 
 JSPS KAKENHI Grant Number JP18K18004, Japan.

\appendix

\section{Fundamental Facts about Matroids 
and Submodular 
Functions}

We briefly give some definitions and fundamental facts 
about matroids, polymatroids, generalized polymatroids, and 
submodular/supermodular functions from a polyhedral point of views, 
which are used in the present paper. 
For general information relevant to the subject of this paper 
see \cite{Frank2011,Fuji2005,Oxley2011,Schrijver2003,Welsh1976}
(the notations used here mostly follow \cite{Fuji2005}).

Let $E$ be a nonempty finite set and ${\bf M}=(E,\mathcal{I})$ be 
a {\it matroid} on $E$ with a family of {\it independent sets} 
(we omit the axioms for independent sets). A maximal independent
set is called a {\it base}. A set $X\subseteq E$ is called a {\it spanning set}
 of ${\bf M}$ if there exists a base $B$ of ${\bf M}$ such that $B\subseteq X$.
A set function $\rho: 2^E\to\mathbb{Z}_{\ge 0}$ defined by 
\begin{equation}\label{eq:p1}
 \rho(X)=\max\{|Y|\mid Y\subseteq X,\ Y\in\mathcal{I}\}
\end{equation}
is called the {\it rank function} of ${\bf M}$. The rank function $\rho$ 
satisfies the {\it submodularity} inequalities
\begin{equation}\label{eq:sf}
  \rho(X)+\rho(Y)\ge \rho(X\cup Y)+\rho(X\cap Y)
\qquad (\forall X, Y \subseteq E).
\end{equation}
Matroid ${\bf M}$ is uniquely determined by each of 
the family of independent sets, the family of bases, the family of 
spanning sets, and the rank function, associated with ${\bf M}$. 
The family of complements $E\setminus B$ of all bases $B$ of ${\bf M}$
is the family of bases of a matroid on $E$, which is called the 
{\it dual} matroid of ${\bf M}$ is denoted by ${\bf M}^*$. 
For the rank function $\rho$ of ${\bf M}$
we denote the rank function of the dual matroid ${\bf M}^*$ by $\rho^*$.
The dual rank function $\rho^*$ is given by
\begin{equation}\label{eq:p2}
 \rho^*(X)=|X|-\rho(E)+\rho(E\setminus X)\qquad (\forall X\subseteq E).
\end{equation}

Any set function $f: 2^E\to\mathbb{R}$ is called a {\it submodular 
function} if it satisfies the submodularity inequalities (\ref{eq:sf})
with $\rho$ being replaced by $f$. The negative of a submodular function
is called a {\it supermodular function}.
Given a submodular function $f: 2^E\to\mathbb{R}$ with $f(\emptyset)=0$, 
the {\it submodular 
polyhedron} associated with $f$ is defined by
\begin{equation}\label{eq:p3}
  {\rm P}(f)=\{x\in\mathbb{R}^E\mid \forall X\subseteq E: x(X)\le f(X)\},
\end{equation}
where $x(X)=\sum_{e\in X}x(e)$. 
(When ${\rm P}(f)\cap\mathbb{R}_{\ge 0}^E\neq\emptyset$, it is called a {\it 
polymatroid} and there uniquely exists a monotone nondecreasing submodular
function $f'$ such that ${\rm P}(f)\cap\mathbb{R}_{\ge 0}^E=
{\rm P}(f')\cap\mathbb{R}_{\ge 0}^E$.)
Also the {\it base polyhedron} associated with $f$ is defined by
\begin{equation}\label{eq:p4}
  {\rm B}(f)=\{x\in{\rm P}(f)\mid x(E)=f(E)\}.
\end{equation}
In a dual manner, given a supermodular function $g: 2^E\to\mathbb{R}$
with $g(\emptyset)=0$, 
the {\it supermodular polyhedron} associated with $g$ is defined by
\begin{equation}\label{eq:p5}
  {\rm P}(g)=\{x\in\mathbb{R}^E\mid \forall X\subseteq E: x(X)\ge g(X)\}
\end{equation}
and the associated base polyhedron by
\begin{equation}\label{eq:p6}
  {\rm B}(g)=\{x\in{\rm P}(g)\mid x(E)=g(E)\}.
\end{equation}
For a submodular function $f: 2^E\to\mathbb{R}$ with $f(\emptyset)=0$
the {\it dual} supermodular 
function $f^\#: 2^E\to\mathbb{R}$ is defined by
\begin{equation}\label{eq:p7}
 f^\#(X)=f(E)-f(E\setminus X)\qquad (\forall X\subseteq E).
\end{equation}
We have ${\rm B}(f)={\rm B}(f^\#)$. Note that $(f^\#)^\#=f$.

For a submodular function $f: 2^E\to\mathbb{R}$ and a supermodular function
$g: 2^E\to\mathbb{R}$ with $f(\emptyset)=g(\emptyset)=0$, if we have
\begin{equation}\label{eq:p8}
  f(X)-g(Y)\ge f(X\setminus Y)-g(Y\setminus X)\qquad (\forall X, Y\subseteq E),
\end{equation}
then the polyhedron $Q(f,g)\equiv{\rm P}(f)\cap{\rm P}(g)$ is called 
a {\it generalized polymatroid}. Every polymatroid is a generalized 
polymatroid.

When $f$ and $g$ are integer-valued, 
all the polyhedra ${\rm P}(f)$, ${\rm P}(g)$, ${\rm B}(f)$, and $Q(f,g)$
are integral. Moreover, given another integer-valued 
submodular $f'$ and supermodular $g'$,
the intersections ${\rm P}(f)\cap{\rm P}(f')$, ${\rm P}(g)\cap{\rm P}(g')$, 
${\rm B}(f)\cap{\rm B}(f')$, and $Q(f,g)\cap Q(f',g')$, if nonempty, are 
integral polyhedra.

For any generalized polymatroid $Q(f,g)$, letting $\hat{e}$ be a new element
and putting $\hat{E}=E\cup\{\hat{e}\}$, for an arbitrary $t\in\mathbb{R}$
define $\hat{f}:\hat{E}\to\mathbb{R}$ by $\hat{f}(\hat{E})=t$ and  
\begin{equation}\label{eq:p9}
  \hat{f}(X)=\left\{
   \begin{array}{ll}
     f(X) & {\rm if\ } \hat{e}\notin X \\
     g(\hat{E}\setminus X) & {\rm if\ } \hat{e}\in X
   \end{array}
   \right.
   \quad (\forall X\subset\hat{E}).
\end{equation}
Then $\hat{f}$ is a submodular function and the projection of the base 
polyhedron ${\rm B}(\hat{f})\subset \mathbb{R}^{\hat{E}}$ 
along the axis $\hat{e}$ into the 
coordinate subspace $\mathbb{R}^E$ is a generalized polymatroid $Q(f,g)$. 
Every generalized polymatroid is obtained in this way and vice versa.
This is an isomorphic correspondence.

For a submodular function $f$, a supermodular function $g$, and 
vectors $l\in(\mathbb{R}\cup\{-\infty\})^E$ and 
$u\in(\mathbb{R}\cup\{+\infty\})^E$ with $l(e)\le u(e)$ for all $e\in E$
we have the following three:
\begin{description}
\item[Fact 1.]
 ${\rm P}(f)^u\equiv\{x\in{\rm P}(f)\mid x\le u\}$ is a submodular 
polyhedron. 
\item[Fact 2.] ${\rm P}(g)_l\equiv\{x\in{\rm P}(g)\mid x\ge l\}$ is a 
supermodular polyhedron.  
\item[Fact 3.] ${\rm B}(f)_l^u\equiv\{x\in{\rm B}(f)\mid l\le x\le u\}$, if nonempty,
is a base polyhedron. 
(In particular, this implies that for a generalized polymatroid $Q(f,g)$ and 
$\alpha, \beta\in\mathbb{R}$ with $\alpha\le \beta$,\  
$Q(f,g)_\alpha^\beta\equiv\{x\in Q(f,g)\mid \alpha\le x(E)\le \beta\}$,
if nonempty, is a generalized polymatroid, due to the isomorphic  
correspondence between base polyhedra and generalized polymatroids.)
\end{description}
These polyhedra are integral if $f$ and $g$ are 
integer-valued and finite $l(e)$s and $u(e)$s are integers.

For any $\alpha\ge 0$, defining $f_{-\alpha}(E)=f(E)-\alpha$ and 
$f_{-\alpha}(X)=f(X)$ for all $X\in 2^E\setminus\{E\}$, we get 
another submodular function, which we call an {\it $\alpha$-truncation} of $f$.
In a dual manner, defining $g_{+\alpha}(E)=g(E)+\alpha$ and 
$g_{+\alpha}(X)=g(X)$ for all $X\in 2^E\setminus\{E\}$, we get 
another supermodular function, called an {\it $\alpha$-truncation} of $g$.
(Note that truncations can be interpreted as operations on generalized 
polymatroids 
(cf.~the above-mentioned Fact 3)).

For any family $\mathcal{F}$ of subsets of $E$ denote by 
${\rm Conv}(\mathcal{F})$ the convex hull of characteristic vectors 
$\chi_X\in\mathbb{R}^E$ for all $X\in\mathcal{F}$, where 
$\chi_X(e)=1$ if $e\in X$ and $=0$ if $e\in E\setminus X$.

Let ${\bf M}=(E,\mathcal{I})$ be a matroid with a rank function $\rho$.
Then we have
\begin{equation}\label{eq:p10}
  {\rm Conv}(\mathcal{I})= {\rm P}(\rho)\cap [0,1]^E,
\end{equation}
which is called a {\it matroid polytope} and denoted by ${\rm P}_{(+)}(\rho)$.
Let $\mathcal{S}$ be the set of spanning sets of ${\bf M}$. Then,
\begin{equation}\label{eq:p11}
  {\rm Conv}(\mathcal{S})= {\rm P}(\rho^\#)\cap [0,1]^E,
\end{equation}
where $\rho^\#$ is the dual supermodular function of $\rho$. 
Define $\bar{\mathcal{I}}=\{E\setminus X\mid X\in\mathcal{I}\}$, which 
is the family of {\it co-spanning} sets of ${\bf M}$, i.e., 
the family of spanning sets of the dual matroid ${\bf M}^*$.
Then we have
\begin{equation}\label{eq:p12}
  {\rm Conv}(\bar{\mathcal{I}})= {\rm P}((\rho^*)^\#)\cap [0,1]^E,
\end{equation}
where $\rho^*$ is the rank function of the dual matroid ${\bf M}^*$.
It follows from (\ref{eq:p2}) and (\ref{eq:p7}) that 
\begin{equation}\label{eq:p13}
  (\rho^*)^\#(X)=|X|-\rho(X) \qquad (\forall X\subseteq E).
\end{equation}
Finally, for any positive integer $k$ define 
\begin{equation}\label{eq:p14}
  \mathcal{I}^k = \{X_1\cup\cdots\cup X_k\mid 
    \forall j\in\{1,\cdots,k\}: X_j\in\mathcal{I}\},
\end{equation}
where note that imposing the condition 
that $X_j\in\mathcal{I}$ $(j=1,\cdots,k)$ are disjoint gives the same 
$\mathcal{I}^k$. The pair $(E,\mathcal{I}^k)$ is a matroid, called 
a {\it union matroid} of $k$ copies of ${\bf M}$, which we denote by
${\bf M}^k$. The rank function $\rho^k$ of ${\bf M}^k$ is given by
\begin{equation}\label{eq:p15}
 \rho^k(X)=\min\{|E\setminus X|+k\rho(X)\mid X\subseteq E\}.
\end{equation}

\end{document}